\documentclass{amsart}
\usepackage{amssymb}
\usepackage{amsmath}
\usepackage{tikz-cd}
\usepackage{amsthm}
\theoremstyle{plain}

\theoremstyle{plain}
\newtheorem{thm}{Theorem}
\newtheorem{prop}{Proposition}
\newtheorem{lem}{Lemma}
\newtheorem{coro}{Corollary}

\usepackage{geometry}
\theoremstyle{definition}
\newtheorem{exam}{Example}

\begin{document}
\setcounter{page}{1}

\title{ On the lattice structure of the space of all Bochner integrable Banach lattice-valued functions}
\author[Omid Zabeti ]{Omid Zabeti}

\address{ Department of Mathematics, Faculty of Mathematics,Statistics and computer science, University of Sistan and Baluchestan, P.O. Box: 98135-674, Zahedan, Iran.}
\email{{o.zabeti@gmail.com}}



\date{Received: xxxxxx; Revised: yyyyyy; Accepted: zzzzzz.
\newline \indent $^{*}$ Corresponding author}

\begin{abstract}
Suppose $(X,\Sigma,\mu)$ is a finite measure space, $E$ is a Banach lattice, and $B(X,E,\mu)$ is the space of all Bochner integrable $E$-valued functions.
In this note, we show that $B(X,E,\mu)$ is a $KB$-space or has the sequential Fatou property if and only if so is $E$. Among this, some results about Bochner integral convergence in $B(X,E,\mu)$, using order structure of $E$, have been proved, as well.
%consider unbounded order convergence in $B(X,E,\mu)$ instead of almost everywhere pointwise convergence to establish two results similar to the monotone convergence theorem and the Fatou's lemma; this approach has an advantage: no monotonicity is required. Moreover, as an application, we show that $B(X,E,\mu)$ has the sequential Fatou property, is a $KB$-space, or is a reflexive one if and only if $E$ possesses the corresponding property.
\end{abstract}
\maketitle
\section{motivation and preliminaries}
Let us start with some motivation.
Bochner integral is one of the best ways to  generalize  the usual notion of integral ( Riemann integral or Lebesgue one) to vector-valued functions.
So, it is an enlightening line for research to discover different aspects of the space of all Bochner integrable functions. Certainly, investigative properties depend on the structure of the underlying vector space; for example, we can not expect order structure in the space of all Bochner integrable $E$-valued functions when $E$ is a Banach space, in general.

 Let $(X,\Sigma,\mu)$ be a finite measure space and $E$ be a Banach lattice. The space of all Bochner integrable $E$-valued functions from $X$ into $E$ is denoted by $B(X,E,\mu)$. It is shown in \cite[Theorem 2.5]{Willem}, it is a Banach lattice under norm $\|f\|=\int_{X}\|f(w)\|d\mu$. For more expositions on this topic, see \cite{Ryan, Willem}.

 In this paper, we investigate whether $B(X,E,\mu)$ is a $KB$-space or is sequentially Fatou. More precisely, we formally prove this holds exactly when $E$ possesses the same property.

 Let us again consider some motivation.
Suppose that $E$ is a Banach lattice. It is known that in some classical function spaces such as $\ell_p$ for $1\leq p\leq\infty$ or $c_0$, unbounded order convergence ( $uo$-convergence, in brief) for nets is as the same as pointwise convergence; in addition, in $L^p(\mu)$-spaces, $uo$-convergence for sequences and almost everywhere convergence agree. So, $uo$-convergence is a generalization of coordinate-wise convergence in general Banach lattices; furthermore, for order bounded nets, $uo$-convergence and order one coincide.
So, it is of independent interest to investigate some theorems such as the monotone convergence theorem which relies on the pointwise convergence, by means of order convergence. Beside this, some attempts have been made to generalize such theorems to monotone Banach lattice-valued functions using pointwise convergence ( recall that in some classical spaces, order convergence and coordinate-wise one agree); see \cite[Proposition 2.6]{Willem}.
In this paper, we try to generalize this result using unbounded order and order convergences.
 For more details regarding Banach lattices, Bochner integral, and the related topics, see \cite{AB, Ryan, Willem}.

Suppose $E$ is a Banach lattice. Recall that for a net $(x_\alpha)$ in $E$, if there is a net $(u_\gamma)$, possibly over a
different index set, with $u_\gamma \downarrow 0$ and for every $\gamma$ there exists $\alpha_0$ such
that $|x_{\alpha} - x| \leq u_\gamma$ whenever $\alpha \geq \alpha_0$, we say that $(x_\alpha)$ converges to $x$ in {\bf order}, in notation, $x_\alpha \xrightarrow{o}x$. Moreover, $(x_{\alpha})$ {\bf unbounded order} converges to $x$ ( $x_{\alpha}\xrightarrow{uo} x$) provided that $|x_{\alpha}-x|\wedge w\xrightarrow{o}0$ for each $w\in E_{+}$. For ample information regarding this type of convergence and the related expositions, see \cite{GTX, GX}.
Recall that a subset $A$ of a Banach lattice $E$ is said to be {\bf almost order bounded} if for each $\varepsilon>0$ there is a positive $u\in E$ such that $A\subseteq [-u,u]+\varepsilon B_E$, in which $B_E$ presents the closed unit ball of $E$. In addition,
  a Banach lattice $E$ is {\bf order continuous} if $x_{\alpha}\downarrow 0$ implies that $\|x_{\alpha}\|\downarrow 0$; it is {\bf $\sigma$-order continuous} provided that $x_n\downarrow 0$ results in $\|x_n\|\downarrow 0$. $E$ is called a {\bf $KB$-space} if every increasing norm bounded sequence in $E_{+}$, converges. Also, the norm of a Banach lattice $E$ is said to be {\bf sequentially Fatou} or we say $E$ has the {\bf sequential Fatou property}  whenever for every increasing sequence $(x_n)\subseteq E_{+}$, $x_n\uparrow x$ implies that $\|x_n\|\uparrow \|x\|$. For more details, see \cite{AB}.
  %Finally, recall that a Banach lattice $E$ is reflexive if and only if $E$ and its dual are $KB$-spaces. 

\section{main results}
First of all, let us start with an elementary but useful fact.
\begin{lem}\label{60}
Let $(X,\Sigma,\mu)$ be a finite measure space and $E$ be a Banach lattice. Then there exists a topologically lattice isomorphism from $E$ into $B(X,E,\mu)$.
\end{lem}
\begin{proof}
Define $T$ from $E$ into $B(X,E,\mu)$ via $T(x)=f_x$ in which $f_x:X\to E$ is determined by $f_x(w)=x$. $f_x$ is a simple function so that strongly measurable. Also, $\|f_x\|=\int_{X}\|f_x(w)\|d\mu=\|x\|\mu(X)$. This shows that $T$ is topologically isomorphism. Indeed, $T(|x|)=|T(x)|$. This would complete the proof.
\end{proof}
The preceding lemma is fruitful in the sense that when $B(X,E,\mu)$ has a property which holds on the lattice and topological structures, we can conclude that $E$ has the same property, too. The first application is included in the next lemma.
A variant of the following result has been obtained in \cite[Page 3, Proposition]{DC}.
\begin{lem}\label{9}
Let $(X,\Sigma,\mu)$ be a finite measure space and $E$ be a Banach lattice. Then $E$ is order continuous if and only if so is $B(X,E,\mu)$.
\end{lem}
\begin{proof}
The direct implication has proved in \cite[Page 3, Proposition]{DC}; since order continuity of a Banach lattice is equivalent to the weak compactness of order intervals.  For the other side, use Lemma \ref{60}.
%Suppose $(x_n)$ is a disjoint order bounded sequence in $E$. Define $f_n:X\to E$ via $f_n(w)=x_n$. Each $f_n$ is a simple function so that strongly measurable. Moreover, $\int_{X}\|f_n(w)\|d\mu=\|x_n\|\mu(X)<\infty$. This means that $(f_n)\subseteq B(X,E,\mu)$. In addition, $(f_n)$ is order bounded and disjoint. This implies that $\|f_n\|=\int_{X}\|f_n(w)\|d\mu\rightarrow 0$. That is $\|x_n\|\rightarrow 0$.
\end{proof}
Now, we establish a $\sigma$-order continuous version of Lemma \ref{9}.
\begin{lem}\label{1001}
Let $(X,\Sigma,\mu)$ be a finite measure space and $E$ be a Banach lattice. Then $E$ is $\sigma$-order continuous if and only if so is $B(X,E,\mu)$.
\end{lem}
\begin{proof}
Suppose $E$ is $\sigma$-order continuous and $(f_n)$ is a sequence in $B(X,E,\mu)$ with $f_n\downarrow 0$. This means that for a.e. $w\in X$, we have $f_n(w)\downarrow 0$. Therefore, by assumption, $\|f_n(w)\|\rightarrow 0$. Now, by the monotone convergence theorem, we have $\|f_n\|=\int_{X}\|f_n(w)\|d\mu\rightarrow 0$. For the other side, one can apply Lemma \ref{60}.
\end{proof}
Let us recall \cite[Proposition 2.6]{Willem}.
\begin{prop}\label{20}
Assume that $(X,\Sigma, \mu)$ is a  measure space and $E$ is a $\sigma$-order continuous Banach lattice. Suppose $(f_n)$ is an increasing norm bounded sequence of Bochner integrable functions such that $f_n\rightarrow f$ for a.e. $w\in X$. Then $f$ is also Bochner integrable and  $\int_{X}f_n d\mu\rightarrow \int_{X}f d\mu$.
\end{prop}
%\begin{proof}
%By \cite[Lemma 1.2]{KMT}, we conclude that $f_n(w)\uparrow f(w)$ for a.e $w\in X$. Now, observe that Proposition \ref{30} yields the desired result.
%, $f$ is also Bochner integrable . In addition, By Lemma \ref{9}, observe that $\int_{X}\|f-f_n\|d\mu\rightarrow 0$. This results in $\int_{X}f_n d\mu\rightarrow \int_{X}f d\mu$.
%\end{proof}
%Consider this point that in Proposition \ref{11}, we can add the extra hypothesis of monotonicity of the sequence $(f_n)$ and instead, use the weaker condition a.e order convergence in $E$, alternatively.

%A variant of the latter consequence ( with weaker conclusion and a bit different method) has been obtained in \cite[Proposition 2.6]{Willem}, independently. %So, the previous result can be considered as a generalization of it, too.
In Proposition \ref{20}, $\sigma$-order continuity is essential and can not be removed. Consider the following.
\begin{exam}\label{5}
Suppose $X=[0,1]$ with the Lebesgue measure and $E=\ell_{\infty}$. For each $n\in \Bbb N$, consider $f_n:X\to E$ via $f_n(t)=(t,\ldots,t,0,\ldots)$, in which $t$ is appeared $n$ times. Clearly, $f_n(t)\uparrow f(t)$, in which $f(t)$ is the constant sequence $t$. Suppose $\alpha_n=\int_{X}(f_n-f)(t)d\mu$. Assume that $\pi_i$ is the linear operator on $E$ defined via $\pi_i((x_n))=x_i$. Certainly, for each $n$, there is an $i$ with $i>n$. Since $\pi_i$ is continuous, by using Bochner integral properties, we have
\[\pi_i(\alpha_n)=\int_{X}\pi_i(0,\ldots,0,t,\ldots)d\mu=\int_{0}^{1}t dt=\frac{1}{2}.\]
So that $\alpha_n\nrightarrow 0$.
%Consider this point that the right hand is the Riemann integral which will be equal to the Lebesgue integral.
\end{exam}
Now, we extend Proposition \ref{20}.
\begin{thm}\label{1000}
Suppose $(X,\Sigma,\mu)$ is a finite measure space and $E$ is a $\sigma$-order continuous Banach lattice. In addition, assume that $(f_n)$ is an almost order bounded sequence of Bochner integrable functions from $X$ into $E$ such that $f_n\xrightarrow{uo}f$. Then, $f$ is Bochner integrable and $\int_{X}f_n d\mu\rightarrow \int_{X}f d\mu$.
\end{thm}
\begin{proof}
First note that by assumption, for a.e. $w\in X$, $f_n(w)\xrightarrow{uo}f(w)$. By \cite[Proposition 3.7]{GX}, we observe that $f_n(w)\rightarrow f(w)$ in norm. Thus, by \cite[Theorem 1.26]{Willem}, we conclude $f$ is also strongly measurable. Since the sequence $(f_n)$ is bounded, for each fixed $w\in X$, by \cite[Lemma 3.6]{GX}, we have $\|f(w)\|\leq \liminf \|f_n(w)\|<\infty$. Using the Fatou's lemma, yields the following.
\[\int_{X}\|f(w)\|d\mu\leq \int_{X}\liminf\|f_n(w)\|d\mu\leq \liminf\int_{X}\|f_n(w)\|d\mu<\infty.\]
So, we conclude Bochner integrability of $f$. Consider this point that by Lemma \ref{1001}, $B(X,E,\mu)$ is also $\sigma$-order continuous so that another application of \cite[Proposition 3.7]{GX} yields the following.
\[\|f_n-f\|=\int_{X}\|(f_n-f)(w)\|d\mu\rightarrow 0.\]
\end{proof}
Furthermore, when we utilize order convergence instead of $uo$-convergence, we can omit almost order boundedness assumption.
\begin{coro}\label{30}
Suppose $(X,\Sigma,\mu)$ is a finite measure space and $E$ is a $\sigma$-order continuous Banach lattice. In addition, assume that $(f_n)$ is a bounded sequence of Bochner integrable functions from $X$ into $E$ such that  $f_n\xrightarrow{o}f$. Then, $f$ is Bochner integrable and $\int_{X}f_n d\mu\rightarrow \int_{X}f d\mu$.
\end{coro}
%\begin{proof}
%By Theorem \ref{1000},
%\end{proof}
Furthermore, Almost order boundedness in Theorem \ref{1000} is also necessary and can not be omitted as the following example shows.
\begin{exam}
Suppose $X=[0,1]$ with the Lebesgue measure and $E=L^1[0,1]$. Assume that $f_n(t)=\delta_n$ in which $\delta_n$ is defined as follows: on the interval $[0,\frac{1}{n}]$, $n$ and zero otherwise. Clearly, $\|\delta_n\|=1$. Observe that $\delta_n\wedge {\sf 1}\downarrow 0$ and the constant function ${\bf 1}$ is a weak unit for $E$ so that by \cite[Corollary 3.5]{GTX} $\delta_n\xrightarrow{uo}0$. But, $\int_{X}f_nd\mu=\int_{X}\delta_nd\mu=1$. Note that $(\delta_n)$ is not almost order bounded since it is not uniformly integrable ( see \cite[Theorem 5.2.9]{AK} for more details).
\end{exam}
\begin{thm}
Let $(X,\Sigma,\mu)$ be a finite measure space and $E$ be a Banach lattice. Then $E$ is a $KB$-space if and only if so is  $B(X,E,\mu)$.

\end{thm}
\begin{proof}
Suppose $E$ is a $KB$-space and $(f_n)$ is an increasing norm bounded positive sequence in $B(X,E,\mu)$. This means that $\int_{X}\|f_n(w)\|d\mu\leq M$; for some $M\in {\Bbb R}_{+}$. Note that the sequence $(\int_{X}\|f_n(w)\|d \mu)$ is also increasing and bounded so that converges to some $\alpha\in \Bbb R$. This results in $\int_{X}(\|f_n(w)\|-\frac{\alpha}{\mu(X)})d\mu\rightarrow 0$. Note that the sequence $(\|f_n\|)\subseteq L^1(\mu)$ is increasing  and norm convergent so that order convergent to the same limit. This, in turn, implies that it is a.e. convergent. Thus, $(f_n)$ is a.e pointwise convergent so that pointwise bounded. Therefore, $(f_n(w))$ is a bounded increasing sequence in $E$. By assumption, it converges. Assume $f_n(w)\rightarrow \alpha_w$. Define $f:X\to E$ via $f(w)=\alpha_w$. Since, $f$ is the pointwise limit of $(f_n)$, it is strongly measurable. In addition, since $f_n(w)\uparrow f(w)$ for a.e $w\in X$, by Proposition \ref{20}, $f$ is also Bochner integrable. Now, by the usual monotone convergence theorem,
$\|f_n-f\|=\int_{X}\|(f_n-f)(w)\|d\mu\rightarrow 0$, as claimed.

For the converse, again, consider Lemma \ref{60}.
%assume that $(x_n)$ is an increasing norm bounded sequence in $E$. Define: $f_n:X\to E$ via $f_n(w)=x_n$. Since each $f_n$ is a simple function, it is strongly measurable. Also, from $\int_{X}\|f_n\|dw=\|x_n\|\mu(X)<\infty$, we see that $f_n\in B(X,E,\mu)$. Furthermore, $(f_n)$ is also increasing and norm bounded. By hypothesis, it is convergent. So, there exists $f\in B(X,E,\mu)$ such that $\|f_n-f\|=\int_{X}\|(f_n-f)(w)\|d\mu\rightarrow 0$. This means that $\|x_n\|\rightarrow \frac{\int_{X}\|f(w)\|d\mu}{\mu(X)}$, as desired.
\end{proof}
%\begin{rem}
%Compatible with \cite[Theorem 4.7]{GX}, we conclude that when $E$ is a $KB$-space, every norm bounded $uo$-Cauchy net $(f_{\alpha})$ in $B(X,E,\mu)$ is $uo$-convergent. That is its $uo$-limit is also Bochner integrable, namely $f$. Moreover, if the net is almost order bounded, then by Theorem \ref{1} $\int_{X}f_{\alpha}d\mu\rightarrow \int_{X}fd\mu$.
%\end{rem}

\begin{thm}\label{50}
Let $(X,\Sigma,\mu)$ be a finite measure space and $E$ be a Banach lattice. Then the norm of $E$ is sequentially Fatou if and only if so is  $B(X,E,\mu)$.
\end{thm}
\begin{proof}
Suppose $E$ has sequentially Fatou property and $(f_n)$ is a positive sequence in $B(X,E,\mu)$ with $f_n\uparrow f$. This means that for a.e. $w\in X$, $f_n(w)\uparrow f(w)$. By assumption, $\|f(w)\|=\sup\{\|f_n(w)\|\}$. On the other hand, $\|f_n(w)\|$ is also a bounded increasing sequence in reals such that $\|f_n(w)\|\rightarrow \|f(w)\|$; in another way, $\|f(w)\|=\sup\{\|f_n(w)\|\}=\lim{\|f_n(w)\|}$. Thus, by the usual monotone convergence theorem, $\int_{X}\|f_n(w)\|d\mu\rightarrow\int_{X}\|f(w)\|d\mu$. This means that
\[\int_{X}\|f(w)\|d\mu=\int_{X}\sup\{\|f_n(w)\|\}d\mu=\sup\{\int_{X}\|f_n(w)\|d\mu\}.\]
Thus, $\|f\|=\sup\{\|f_n\|\}$, as desired.

For the converse, one can utilize Lemma \ref{60}.
%Assume that the norm of $B(X,E,\mu)$ is sequentially Fatou and $(x_n)$ is a positive sequence in $E$ with $x_n\uparrow x$. Define $f_n:X\to E$ via $f_n(w)=x_n$. Each $f_n$ is a simple function so that strongly measurable. On the other hand, $\int_{X}\|f_n(w)\|d\mu=\|x_n\|\mu(X)<\infty$; observe that $(x_n)$ is order bounded so that bounded. Thus, the sequence $(f_n)$ is Bochner integrable. In  a similar manner, $f(w)=x$ is also Bochner integrable and one may verify that $f_n\uparrow f$ in $B(X,E,\mu)$. By assumption, $\|f\|=\sup\{\|f_n\|\}$. This implies that $\|x_n\|\uparrow \|x\|$.
\end{proof}
{\bf Question}. We do not know whether Theorem \ref{50} is valid if the norm is Fatou. Recall that the norm of a Banach lattice $E$ is called Fatou if for every positive net $(x_{\alpha})$ with $x_{\alpha} \uparrow x$, we have $\|x_{\alpha}\| \uparrow \|x\|$. Of course, when $B(X,E,\mu)$ is Fatou, the above argument may apply to show that so is $E$ but for the other side, the above trick would not work. Nevertheless, we have seen in Lemma \ref{9} that under a bit stronger condition "order continuity" both directions hold.
\begin{thm}
Suppose $E$ is a Banach lattice. If the dual of $B(X,E,\mu)$ is order continuous, then so is the dual of $E$.
\end{thm}
\begin{proof}
Suppose the dual of $B(X,E,\mu)$ is order continuous but $E'$ is not. So, $E$ contains a lattice copy of $\ell_1$. By Lemma \ref{60}, we conclude that $B(X,E,\mu)$ also contains a lattice copy of $\ell_1$ which contradicts our hypothesis.
\end{proof}
Note that converse of the preceding theorem is not true, in general. Put $X=[0,1]$ with the Lebesgue measure and $E=\Bbb R$. Indeed, $E'$ is order continuous but certainly not the dual of $L^1[0,1]$.


\begin{thebibliography}{1}
%\bibitem{AB1} C. D. Aliprantis and  K. C. Border, {\em Infinite dimensional analysis,} a Hichhiker's guide, Springer, third edition, 2006.
\bibitem{AB} C. D.  Aliprantis and O. Burkinshaw, Positive operators, Springer, 2006.
\bibitem{AK} F. Albiac and N. J. Kalton, {\em Topics in Banach space theory,} Graduate Texts in Mathematics, 233, Springer, New York, 2006.
\bibitem{DC} D. I. Cartwright, {\em The order completeness of some
spaces of vector-valued functions,} Bull. Austral. Math. Soc, {\bf 11}(1974), pp. 57--61.
\bibitem{GTX} N. Gao, V. G. Troitsky, and F. Xanthos, {\em Uo-convergence and its applications to Cesàro means in Banach lattices,}  Israel J. Math., {\bf 220} (2017), pp. 649--689.
%\bibitem{DOT} Y. Deng, M O'Brien, and V. G. Troitsky, {\em Unbounded norm convergence in Banach lattices,} Positivity., {\bf 21}(3) (2017), pp. 963--974.
\bibitem{GX} N. Gao and F. Xanthos, {\em Unbounded order convergence and application
to martingales without probability,} J. Math. Anal. Appl., {\bf 415} (2014), pp
931--947.
%\bibitem{KMT} M. Kandi\'{c}, M.A.A. Marabeh, and  V. G. Troitsky, {\em Unbounded norm topology in Banach lattices,} J. Math. Anal. Appl., {\bf 451(1)} (2017), pp. 259--279.
 %\bibitem{RZ} A. C. M. van Rooij and  W. van Zuijlen, {\em Bochner integrals in ordered vector spaces,} Positivity., {\bf 21}(4) (2017), pp. 1089--1113.
\bibitem{Ryan} R. A. Ryan, {\em Tensor products of Banach spaces,} Springer, 2002.
\bibitem{Willem} W. van Zuijlen, {\em Integration of functions
with values in a Riesz space}, Radboud universiteit Nijmegen, Master thesis, July 2012.


\end{thebibliography}
\end{document}